\documentclass[11pt, english]{article}

\usepackage[margin= 2 cm,bottom=21mm, top= 20mm]{geometry}

\usepackage[dvipsnames]{xcolor}
\usepackage[square,sort,comma,numbers]{natbib}
\usepackage{amsthm}
\usepackage{amsmath}
\usepackage{amssymb}
\usepackage{setspace}
\usepackage{mathtools}
\usepackage{graphicx}
\graphicspath{ {./images/} }
\usepackage[hidelinks]{hyperref}
\usepackage{cleveref}
\usepackage{hyperref}
\usepackage{enumitem}
\usepackage{framed}
\usepackage{subcaption}
\usepackage{xspace}
\usepackage{thmtools} 
\usepackage{thm-restate}
\usepackage{tikz-feynman}
\usepackage{thmtools}
\usepackage{thm-restate}
\usepackage{ifthen} 
\usepackage{floatrow}

\usepackage{tikz}
\usepackage{mathdots}
\usepackage{xcolor}
\usepackage{diagbox}
\usepackage{colortbl}
\usepackage[absolute,overlay]{textpos}

\usetikzlibrary{shapes.misc}
\usetikzlibrary{decorations.pathmorphing}

\tikzset{snake it/.style={decorate, decoration=snake}}
\usetikzlibrary{math}

\usetikzlibrary{calc}
\usetikzlibrary{decorations.pathreplacing}
\usetikzlibrary{positioning,patterns}
\usetikzlibrary{arrows,shapes,positioning}
\usetikzlibrary{decorations.markings}

\tikzstyle{edge}=[very thick]
\definecolor{bostonuniversityred}{rgb}{0.8, 0.0, 0.0} 
\definecolor{arsenic}{rgb}{0.23, 0.27, 0.29}
\tikzstyle{diredge}=[postaction={decorate,decoration={markings,
		mark=at position .95 with {\arrow[scale = 1]{stealth};}}}]
\tikzset{
    arrow/.style={decoration={markings, mark=at position 0.7 with
    {\fill(-0.09*#1,-0.03*#1) -- (0,0) -- (-0.09*#1,0.03*#1) -- cycle;}}, postaction={decorate}},
    arrow/.default=1
}
\tikzset{
    arow/.style={decoration={markings, mark=at position 1 with
    {\fill(-0.09*#1,-0.03*#1) -- (0,0) -- (-0.09*#1,0.03*#1) -- cycle;}}, postaction={decorate}},
    arow/.default=1
}
\tikzset{
    arrrow/.style={decoration={markings, mark=at position 0.9 with
    {\fill(-0.09*#1,-0.03*#1) -- (0,0) -- (-0.09*#1,0.03*#1) -- cycle;}}, postaction={decorate}},
    arow/.default=1
}

\newcommand{\fitellipsis}[2] 
{\draw [fill=white]let \p1=(#1), \p2=(#2), \n1={atan2(\y2-\y1,\x2-\x1)}, \n2={veclen(\y2-\y1,\x2-\x1)}
    in ($ (\p1)!0.5!(\p2) $) ellipse [ x radius=\n2/2+0cm, y radius=1.1cm, rotate=\n1];
}
\newcommand{\Fitellipsis}[2] 
{\draw [fill=white]let \p1=(#1), \p2=(#2), \n1={atan2(\y2-\y1,\x2-\x1)}, \n2={veclen(\y2-\y1,\x2-\x1)}
    in ($ (\p1)!0.5!(\p2) $) ellipse [ x radius=\n2/2+0cm, y radius=1.4cm, rotate=\n1];
}

\theoremstyle{plain}

\newtheorem*{thm*}{Theorem}
\newtheorem{thm}{Theorem}[section]
\Crefname{thm}{Theorem}{Theorems}

\newtheorem*{lem*}{Lemma}
\newtheorem{lem}[thm]{Lemma}
\Crefname{lem}{Lemma}{Lemmas}

\newtheorem*{claim*}{Claim}

\crefname{claim}{Claim}{Claims}
\Crefname{claim}{Claim}{Claims}

\Crefname{prop}{Proposition}{Propositions}

\Crefname{remar}{Remark}{Remarks}

\crefname{cor}{Corollary}{Corollaries}

\newtheorem*{conj*}{Conjecture}

\crefname{conj}{Conjecture}{Conjectures}

\Crefname{qn}{Question}{Questions}

\newtheorem{obs}[thm]{Observation}
\Crefname{obs}{Observation}{Observations}

\Crefname{ex}{Example}{Examples}

\theoremstyle{definition}

\Crefname{prob}{Problem}{Problems}

\newtheorem{defn}[thm]{Definition}
\Crefname{defn}{Definition}{Definitions}

\theoremstyle{remark}

\renewenvironment{proof}[1][]{\begin{trivlist}
\item[\hspace{\labelsep}{\bf\noindent Proof#1.\/}] }{\qed\end{trivlist}}

\newcommand{\remove}[1]{}

\newcommand{\ceil}[1]{
    \left\lceil #1 \right\rceil
}

\newcommand{\eps}{\varepsilon}

\title{\vspace{-0.85 cm}
Chv\'atal-Erd\H{o}s condition for pancyclicity}
\date{}

\author{
Nemanja Dragani\'c\thanks{
Department of Mathematics, ETH, Z\"urich, Switzerland. Research supported in part by SNSF grant 200021\_196965.
\newline
\emph{Emails}: \textbf{\{nemanja.draganic,david.munhacanascorreia, benjamin.sudakov\}@math.ethz.ch}.
}
\and
David Munh\'a Correia\footnotemark[1]
\and
Benny Sudakov\footnotemark[1]}
\begin{document} 
\maketitle
\begin{abstract}
An $n$-vertex graph is \emph{Hamiltonian} if it contains a cycle that covers all of its vertices and it is \emph{pancyclic} if it contains cycles of all lengths from $3$ up to $n$. A celebrated meta-conjecture of Bondy states that every non-trivial condition implying Hamiltonicity also implies pancyclicity (up to possibly a few exceptional graphs). We show that every graph $G$ with $\kappa(G) > (1+o(1)) \alpha(G)$ is pancyclic. This extends the famous Chv\'atal-Erd\H{o}s condition for Hamiltonicity and proves asymptotically a $30$-year old conjecture of Jackson and Ordaz. 
\end{abstract}
\section{Introduction}

The notion of Hamiltonicity is one of most central and extensively studied topics in Combinatorics. Since the problem of determining whether a graph is Hamiltonian is NP-complete, a central theme in Combinatorics is to derive sufficient conditions for this property. A classic example is Dirac’s theorem \cite{dirac1952some} which dates back to 1952 and states that every $n$-vertex graph with minimum degree at least $n/2$ is Hamiltonian. Since then, a plethora of interesting and important results about various aspects of Hamiltonicity have been obtained, see e.g. \cite{ajtai1985first,chvatal1972note,kuhn2013hamilton,krivelevich2011critical,krivelevich2014robust,MR3545109,ferber2018counting, cuckler2009hamiltonian, posa1976hamiltonian}, and the surveys \cite{gould2014recent, MR3727617}.

Besides finding sufficient conditions for containing a Hamilton cycle, significant attention has been given to conditions which force a graph to have cycles of other lengths. Indeed, \emph{the cycle spectrum of a graph}, which is the set of lengths of cycles contained in that graph, has been the focus of study of numerous papers and in particular gained a lot of attention in recent years \cite{liu2020solution, alon2022cycle, friedman2021cycle, liu2018cycle, verstraete2016extremal, alon2021divisible, milans2012cycle, keevash2010pancyclicity, ourpaper}. Among other graph parameters, the relation of the cycle spectrum to the minimum degree, number of edges, independence number, chromatic number and expansion of the graph have been studied.

We say that an $n$-vertex graph is \emph{pancyclic} if the cycle spectrum contains all integers from $3$ up to $n$.
In the cycle spectrum of an $n$-vertex graph, it is usually hardest to guarantee the existence of the longest cycle, i.e. a Hamilton cycle.  This intuition was captured in Bondy's famous meta-conjecture \cite{bondy10pancyclic} from 1973, which asserts that any non-trivial condition which implies Hamiltonicity, also implies pancyclicity (up to a small class of exceptional graphs). As a first example, he proved in \cite{bondy1971pancyclic} an extension of Dirac's theorem, showing that minimum degree at least $n/2$ implies that the graph is either pancyclic or that it is the complete bipartite graph $K_{\frac{n}{2},\frac{n}{2}}$. Further, Bauer and Schmeichel \cite{bauer1990hamiltonian}, relying on previous results of Schmeichel and Hakimi \cite{schmeichel1988cycle}, showed that the sufficient conditions for Hamiltonicity given by Bondy \cite{bondy1980longest}, Chvátal \cite{chvatal1972hamilton} and Fan \cite{fan1984new} all imply pancyclicity, up to a certain small family of exceptional graphs. 

Another classic condition which implies Hamiltonicity is given by the famous theorem of Chvatál and Erd\H{o}s \cite{chvatal1972note}. It states that if the connectivity of a graph $G$ is at least as large as its independence number, that is, $\kappa(G)\geq  \alpha(G)$, then $G$ is Hamiltonian. The pancyclicity counterpart of this result has also been investigated - see, e.g.,  \cite{amar1991pancyclism} and the surveys \cite{jackson1990chvatal, saito2007chvatal}. In fact, in 1990, Jackson and Ordaz \cite{jackson1990chvatal} conjectured that $G$ must be pancyclic if $\kappa(G) > \alpha(G)$, which if true would confirm Bondy's meta-conjecture for this classical instance. One can use an old result of Erd\H{o}s \cite{erdos1972some} to show pancyclicity if $\kappa(G)$ is large enough function of $\alpha(G)$. 
A first linear bound on $\kappa(G)$ was given only in 2010 by Keevash and Sudakov \cite{keevash2010pancyclicity}, who showed that $\kappa(G)\geq 600\alpha(G)$ is enough. In this paper, we resolve the conjecture of Jackson and Ordaz asymptotically, by showing that $\kappa(G) > (1+o(1))\alpha(G)$ is already enough to guarantee pancyclicity.

\begin{thm}\label{thm:main}
Let $\eps >0$ and let $n$ be sufficiently large. Then, every $n$-vertex graph $G$ for which we have $\kappa(G) \geq (1 + \eps) \alpha(G)$ is pancylic.
\end{thm}
\noindent Next we briefly discuss some of the key steps in the proof of this theorem. It will be convenient for us to consider different ranges of cycle lengths whose existence we want to show, and for each range we have a separate subsection which deals with it. This is done in Section \ref{sec:thm}. In order to find these different cycle lengths we will combine various tools on shortening/augmenting paths and finding consecutive path lengths between two fixed vertices.  

For example, for finding consecutive path lengths we crucially use that since $\kappa(G) > \alpha(G)$, it must be that $G$ contains triangles - moreover, it contains a \emph{path with triangles attached to many of its edges} (see Definition \ref{def:pathstriangles}), which trivially implies the existence of many consecutive path lengths between the endpoints of such a path. For shortening/augmenting paths, we also introduce new tools. One of them is used to shorten paths using only the minimum degree of the graph (Lemma \ref{lem:jumpwithdegree}), while another one augments paths using both the independence and connectivity number (Lemma \ref{lem:augmentingpath}). Furthermore, we will also use a novel result proven in \cite{ourpaper} using the Gallai-Milgram theorem, in order to shorten paths using the independence number of the graph (Lemma \ref{lem:jumpwithzigzag}). In Section~\ref{sec:preliminaries} we present these tools, together with some other useful results of a similar flavour. After that, in Section \ref{sec:thm}, we prove Theorem~\ref{thm:main}. The general proof strategy is to find a cycle of appropriate length which consists of two paths, one of which has many edges to which triangles are attached. Then we apply our shortening/augmenting results to the other path. Combining the consecutive path lengths from the first path with the path lengths obtained from the second path (see Observation~\ref{obs:combining}), we will get all possible cycle lengths. Finally, in Section \ref{sec:concludingrem} we make some concluding remarks.

\section{Preliminaries}\label{sec:preliminaries}
\subsection{Notation and definitions}
We mostly use standard graph theoretic notation. Let $G$ be a finite graph. Denote by $V(G)$ its vertex set, and let $S_1,S_2\subseteq V(G)$. We denote by $G[S_1]$ the subgraph of $G$ induced by $S_1$, and by $E[S_1,S_2]$ the set of edges with one endpoint in $S_1$ and the other in $S_2$. Let $H$ be a subgraph of $G$. We denote by $G[H]$ the graph $G[V(H)]$.
A path $P=(x_0,x_1,\ldots,x_\ell)$ of length $\ell$ is a graph on vertex set $\{x_0,x_1,\ldots,x_\ell\}$ with an edge between $x_{i-1}$ and $x_{i}$ for all $i\in[\ell]$. We say that $x_0$ and $x_\ell$ are the endpoints of $P$, and we call $P$ an $x_0x_\ell$-path. Given disjoint sets of vertices $A,B$, we say that $P$ is a path \emph{going from $A$ to $B$} if $x_0 \in A, x_l \in B$ and $x_i \notin A \cup B$ for all $0 < i < l$. We denote by $\alpha(G)$ the independence number of $G$. The \emph{connectivity} $\kappa (G)$ of a connected graph $G$ is the minimum number of vertices whose removal makes G disconnected or reduces it to a trivial graph.

Given sets $A_1,A_2\subset \mathbb N$, we denote by $A_1+A_2$ the set of integers $c$ such that $c=a_1+a_2$ for some $a_1\in A_1$ and $a_2\in A_2.$ Throughout the paper we omit floor and ceiling signs for clarity of presentation, whenever it does not impact the argument.
\begin{defn}
Let  $a,b,p$ be positive real numbers. Given a graph $G$, and two vertices $x$ and $y$, we say that the pair $xy$ is $p$-dense in the interval $[a,b]$ if for every subinterval $[a',b']$ with $b'-a'\geq 
p$ there is an integer $\ell\in[a',b']$ and an $xy$-path in $G$ of length $\ell$. Note that, in particular, $xy$ is $0$-dense in $[a,b]$ if there are paths of all lengths in $[a,b]$ between $x$ and $y$.
\end{defn}
\noindent We now give a trivial observation  which will be used in the proof of Theorem \ref{thm:main}. It states that appropriate combinations of internally vertex-disjoint paths of different lengths imply the existence of cycles of many different lengths.
\begin{obs}\label{obs:combining}
Let $G$ be a graph whose vertex set contains $t$ disjoint sets $S_1,\ldots, S_t$ and another set of $t$ vertices $v_1,\ldots,v_t$ outside of $\bigcup_{i=1}^t S_i$. For each $i\in[t]$, let $A_i\subset \mathbb N$ and suppose that for every $i$ the induced subgraph $G[v_i\cup S_i\cup v_{i+1}]$ is such that it contains a $v_iv_{i+1}$-path of length $\ell$ for each $\ell\in A_i$ (with $v_{t+1}=v_1$). Then for every $\ell\in A_1+\ldots+A_t$, the graph $G$ contains a cycle of length $\ell$.
\end{obs}

\subsection{Cycles and paths with triangles}\label{sec:pathstriangles}
One of the crucial objects which are used in our proof will be cycles which have triangles attached to some of their  edges. Evidently, one can increase the length of such a cycle by precisely one, by using the two edges of a triangle, instead of the edge which lies on the cycle. 
\begin{defn}\label{def:pathstriangles}
Define the graph $C^{r}_\ell$ to be the graph formed by a cycle $v_1v_2 \ldots v_lv_1$ of length $\ell$ with the additional edges $v_1v_3,v_3v_5, \ldots, v_{2r-1}v_{2r+1}$ (if $r = 0$, then it is just a cycle of length $l$). We will refer to this as a \emph{cycle of length $\ell$ with $r$ triangles}. Similarly define $P^{r}_\ell$ and refer to it as a \emph{path of length $l$ with $r$ triangles}, where $P_0^0$ is just a vertex.
\end{defn}
\noindent The following is an easy starting point for the existence of the graphs $C^{r}_\ell$ with appropriate parameters, as subgraphs in graphs $G$ with  $\kappa(G) \geq \alpha(G)$.
\begin{lem}\label{lem:cyclewithtriangles}
Every $n$-vertex graph $G$ with $\kappa(G) \geq \alpha(G)$ contains a $C^{r}_l$ for all $r$ such that $0 \leq r \leq \frac{\kappa(G) - \alpha(G)}{2}$ and some $l$ with $l - 2(r+1) \leq \max \left(\frac{n}{\kappa(G)-2r+1}, \frac{n}{\kappa(G)-1}\right)$. In particular, it contains a $P^{r}_{2r}$ for all such $r$.
\end{lem}
\begin{proof}
We will first show that $G$ must always contain a $P^{r'}_{2r'}$ for $r':=\left\lfloor \frac{\kappa(G) - \alpha(G)}{2}\right\rfloor$- we construct such a path greedily. Suppose that we have the vertices $v_1v_2v_3 \ldots v_{2i+1}$ which form a $P^{i}_{2i}$, so that the edges $v_1v_3, \ldots, v_{2i-1}v_{2i+1}$ are also present. Provided that $i < r'$, we can augment this path as follows. Consider the set $S := N(v_{2i+1}) \setminus \{v_1, \ldots, v_{2i}\}$ - by assumption, this has size at least $\delta(G) - 2i > \kappa(G) - 2r' \geq \alpha(G)$. Therefore, it must contain an edge $v_{2i+2}v_{2i+3}$. Clearly, $v_{2i+1}v_{2i+2}v_{2i+3}$ forms a triangle and thus, $v_1v_2v_3 \ldots v_{2i+1}v_{2i+2}v_{2i+3}$ is a $P^{i+1}_{2i+2}$. Continuing with this procedure until $i = r'$, gives the desired $P^{r'}_{2r'}$.

Now, fix $r$ with the given condition. If $r = 0$, then take an edge $xy$ in $G$. By Menger's theorem, there exist at least $\kappa(G)$ internally vertex-disjoint $xy$-paths in $G$ and thus, at least $\kappa(G) - 1$ of these are not the edge $xy$. Therefore, there is such a path with at most $\frac{n}{\kappa(G) - 1} + 2$ vertices, which together with the edge $xy$, then creates a cycle of length at most $\frac{n}{\kappa(G) - 1} + 2$. If $r \geq 1$, by the previous paragraph, $G$ contains a $P^{r}_{2r}$ - let $x,y$ be its endpoints. By Menger's theorem, there exist at least $\kappa(G)$ internally vertex-disjoint $xy$-paths in $G$. Since at most $2r-1$ of these intersect $P^{r}_{2r} \setminus \{x,y\}$, there exists one which is disjoint to $P^{r}_{2r} \setminus \{x,y\}$ and contains at most $\frac{n}{\kappa(G) - 2r+1}$ internal vertices. This produces the desired $C^{r}_l$.
\end{proof}
\noindent We can also use this type of cycles to extend the celebrated Chv\'atal-Erd\H{o}s theorem \cite{chvatal1972note}. 

\begin{thm}[Chv\'atal-Erd\H{o}s \cite{chvatal1972note}]
    If for a graph $G$ we have that $\kappa(G)\geq \alpha (G)$, then $G$ is Hamiltonian.
\end{thm}

\noindent Our resut states that if the Chv\'atal-Erd\H{o}s condition is satisfied, then we can find a Hamilton cycle with a certain number of triangles, depending on the discrepancy between the connectivity and the independence number.

\begin{thm}\label{ECwithtriangle}
Every $n$-vertex graph $G$ such that $\kappa(G) \geq \alpha(G)$ contains a $C^r_n$ with $r =\left\lfloor \frac{\kappa(G) - \alpha(G)}{2}\right\rfloor$.
\end{thm}
\begin{proof}
Suppose for contradiction that some $\ell<n$ is maximal such that there exists a copy of $C^{r}_\ell$ in $G$.
Note that $\ell$ exists by Lemma \ref{lem:cyclewithtriangles}. Order the cycle as $v_1v_2 \ldots v_\ell v_1$ so that the edges $v_1v_3,v_3v_5, \ldots, v_{2r-1}v_{2r+1}$ are also present. Since $\ell < n$, there is a vertex $v$ not in $C^{r}_l$. Moreover, as $\kappa(G) \geq \alpha(G) + 2r$, there exist $\alpha(G)$ paths contained in $V(G) \setminus \{v_1, \ldots, v_{2r}\}$, all of which go from $v$ to $C^{r}_l$ and are vertex-disjoint apart from the initial vertex $v$. Let us denote these paths as $P_{i_1}, P_{i_2}, \ldots$ so that $v_j = P_j \cap C^r_l$. Consider the set $S := \{v_{i_1 + 1}, v_{i_2 + 1}, \ldots\}$ with indices taken modulo $l$, so that $|S| \geq \alpha(G)$. Observe (as illustrated in Figure \ref{fig:erdoschvatalswitch}) that then there must be an edge contained in $S \cup \{v\}$ and that any such edge can be used to augment $C^{r}_l$ to a $C^{r}_{l'}$ with $l' > l$, contradicting the maximality of $l$.
\end{proof}

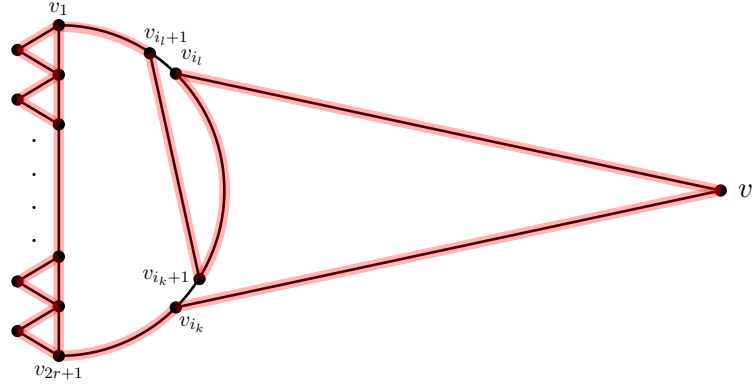
\begin{figure}[h]
    \centering
    \begin{tikzpicture}[scale=1.1,main node/.style={circle,draw,color=black,fill=black,inner sep=0pt,minimum width=3pt}]
        \tikzset{cross/.style={cross out, draw=black, fill=none, minimum size=2*(#1-\pgflinewidth), inner sep=0pt, outer sep=0pt}, cross/.default={2pt}}
	\tikzset{rectangle/.append style={draw=brown, ultra thick, fill=red!30}}

\draw[line width = 1pt] (0,-2) arc
    [
        start angle=-90,
        end angle=90,
        x radius=2cm,
        y radius =2cm
    ] ;
\draw[red, opacity = 0.3, line width = 4pt] (1.1,1.66) arc
    [
        start angle=57,
        end angle=90,
        x radius=2cm,
        y radius =2cm
    ] ;
\draw[red, opacity = 0.3, line width = 4pt] (1.414,-1.414) arc
    [
        start angle=-45,
        end angle=-90,
        x radius=2cm,
        y radius =2cm
    ] ;
    \draw[red, opacity = 0.3, line width = 4pt] (1.7,-1.07) arc
    [
        start angle=-32,
        end angle=45,
        x radius=2cm,
        y radius =2cm
    ] ;
\draw[line width = 1pt] (0,-2) to (0,2);
\draw[red, opacity = 0.3, line width = 4pt] (0,-2) to  (0,2);

\node[scale = 4] at (0,2) {$.$};
\node[scale = 0.75] at (0,2.2) {$v_1$};
\node[scale = 4] at (0,-2) {$.$};
\node[scale = 0.75] at (0,-2.2) {$v_{2r+1}$};

\node[scale = 4] at (0,-1.4) {$.$};
\node[scale = 4] at (0,-0.8) {$.$};
\node[scale = 4] at (0,1.4) {$.$};
\node[scale = 4] at (0,0.8) {$.$};
\node[scale = 4] at (-0.5,-1.7) {$.$};
\node[scale = 4] at (-0.5,-1.1) {$.$};
\node[scale = 4] at (-0.5,1.7) {$.$};
\node[scale = 4] at (-0.5,1.1) {$.$};
\draw[line width = 1pt] (0,-2) to (-0.5,-1.7);
\draw[line width = 1pt] (0,-1.4) to (-0.5,-1.7);
\draw[line width = 1pt] (0,2) to (-0.5,1.7);
\draw[line width = 1pt] (0,1.4) to (-0.5,1.7);
\draw[line width = 1pt] (0,1.4) to (-0.5,1.1);
\draw[line width = 1pt] (0,0.8) to (-0.5,1.1);
\draw[line width = 1pt] (0,-1.4) to (-0.5,-1.1);
\draw[line width = 1pt] (0,-0.8) to (-0.5,-1.1);
\draw[red, opacity = 0.3, line width = 4pt] (0,-2) to (-0.5,-1.7);
\draw[red, opacity = 0.3,line width = 4pt] (0,-1.4) to (-0.5,-1.7);
\draw[red, opacity = 0.3,line width = 4pt] (0,2) to (-0.5,1.7);
\draw[red, opacity = 0.3,line width = 4pt] (0,1.4) to (-0.5,1.7);
\draw[red, opacity = 0.3,line width = 4pt] (0,1.4) to (-0.5,1.1);
\draw[red, opacity = 0.3,line width = 4pt] (0,0.8) to (-0.5,1.1);
\draw[red, opacity = 0.3,line width = 4pt] (0,-1.4) to (-0.5,-1.1);
\draw[red, opacity = 0.3,line width = 4pt] (0,-0.8) to (-0.5,-1.1);
\node at (-0.3,-0.6) {$.$};
\node at (-0.3,-0.2) {$.$};
\node at (-0.3,0.2) {$.$};
\node at (-0.3,0.6) {$.$};

\node[scale = 4] at (8,0) {$.$};
\node at (8.3,0) {$v$};

\node[scale = 4] at (1.414,1.414) {$.$};
\node[scale = 4] at (1.414,-1.414) {$.$};
\node[scale = 4] at (1.1,1.66) {$.$};
\node[scale = 4] at (1.7,-1.07) {$.$};
\node[scale = 0.75] at (1.3,-1.05) {$v_{i_k+1}$};
\node[scale = 0.75] at (1.614,-1.614) {$v_{i_k}$};
\node[scale = 0.75] at (1.614,1.614) {$v_{i_l}$};
\node[scale = 0.75] at (1.3,1.86) {$v_{i_l+1}$};

\draw[line width = 1pt] (1.1,1.66) to (1.7,-1.07);
\draw[line width = 1pt] (8,0) to (1.414,1.414);
\draw[line width = 1pt] (8,0) to (1.414,-1.414);
\draw[red, opacity = 0.3,line width = 4pt] (8,0) to (1.414,1.414);
\draw[red, opacity = 0.3,line width = 4pt] (8,0) to (1.414,-1.414);
\draw[red, opacity = 0.3,line width = 4pt] (1.1,1.66) to (1.7,-1.07);




    \end{tikzpicture}
     \caption{An illustration of how an edge between two elements $v_{i_k+1},v_{i_l+1}$ of $S$ can be used to construct a new $C^{r}_{l'}$.} \label{fig:erdoschvatalswitch}

\end{figure}
\noindent We finish this section with the following partitioning lemma - it will allow us to transform even cycles found by standard density considerations into odd cycles.
\begin{lem}\label{lem:partition}
Let $G$ be an $n$-vertex graph with $\kappa(G) > \alpha(G)$. Then, there exists $X \subseteq V(G)$ and a set of edges $E$ contained in $G[X]$ such that the following hold.
\begin{itemize}
    \item $|E| \geq \left(\frac{\kappa(G) - \alpha(G)}{8} \right) n$.
    \item For every edge $e = xy \in E$ there is a $z \in V(G) \setminus X$ such that $xzy$ is a triangle in $G$.
\end{itemize}
\end{lem}
\begin{proof}
First, since every vertex set in $G$ of size at least $\alpha(G)+1$ contains an edge, every vertex $v$ in $G$ is such that its neighbourhood $N(v)$ contains a matching of size at least $\frac{\delta(G) - \alpha(G)}{2} \geq \frac{\kappa(G) - \alpha(G)}{2}$ - let $r$ denote this quantity. For each $v$, fix such a matching $M_v$.

Now, let $X$ be a random subset of $V(G)$ where each vertex is chosen independently at random with probability $1/2$. Let $E$ denote the set of edges $e = xy$ with the following property: $x,y \in X$ and there is some $z \notin X$ such that $yz \in M_x$ or $xz \in M_y$. Clearly, $E$ satisfies the second condition of the lemma. We need only to estimate the expected value of $|E|$ in order to prove than the first condition is satisfied for some $X$. Indeed, note that for an edge $e = xy$ to be present in $E$ we must have that there is some $z$ such that $yz \in M_x$ or $xz \in M_y$. Further, if at least one of these options holds, it is clear that then $\mathbb{P}(e \in E) \geq \frac{1}{8}$; since that is the probability that $x,y\in X$ and $z\notin X$. To finish, note that the number of such edges is at least $\frac{1}{2}\sum_v 2|M_v| = \sum_v |M_v| \geq nr$.
Indeed, for each vertex $x\in G$, every vertex $y$ in the matching $M_x$, gives such an edge $xy$, but since we possibly double counted ($x$ might be in the matching $M_y$), the total number of such edges is at least $\frac{1}{2}\sum_v 2|M_v|$.
Hence, $\mathbb{E}[|E|] \geq nr/8$, so there must exist such an $E$ with $|E| \geq nr/8$ as desired.
\end{proof}

\subsection{Path shortening/augmenting tools}
In this section, we describe some tools for shortening paths. First, we show the following lemma which uses only the minimum degree of the graph.
\begin{lem}\label{lem:jumpwithdegree}
Let $G$ be an $n$-vertex graph, $\delta := \delta(G)$ and $P$ a path in $G$ with endpoints $x,y$ such that $|P| > 20n/\delta$. Then there is an 
$xy$-path $P'$ such that $|P| - 20n/\delta \leq |P'| < |P|$.
\end{lem}
\begin{proof}
Suppose for sake of contradiction that no such path $P'$ exists. Let $P := v_1 v_2 \ldots v_{l-1} v_l$ with $v_1 = x, v_l = y$ and let $<_P$ denotes the given ordering of the path $P$ as $v_1 <_P v_2 <_P  \ldots <_P v_l$. Since $|P| > 10n/ \delta$, we can partition $P$ into sub-paths $Q_1,Q_2, \ldots, Q_k$ such that $|Q_k| \leq 10n/\delta$ and $|Q_i| = 10n/ \delta$ for all $i < k$. Moreover, we have $k = \left\lceil \frac{|P|}{10n/\delta}\right\rceil$. Now, consider the vertices in $Q_1$ and take a subset $Q'_1 \subseteq Q_1$ of size $\lfloor |Q_1|/3 \rfloor \geq 3n/\delta$ such that no two vertices in $Q'_1$ are at distance at most $2$ in $P$. Consider then the set of edges incident to $Q'_1$, that is, $E[Q'_1,V(G)]$; by the minimum degree condition, there are at least $|Q'_1| \cdot \delta \geq 3n$ such edges.

Now, clearly there cannot exist an edge spanned by $Q_1$ which does not belong to $P$ since this edge could be used to shorten $P$ by at most $|Q_1| \leq 10 n/\delta$. Hence, $e(Q'_1, Q_1) \leq 2|Q'_1|$. Similarly, the following must hold.
\begin{claim*}
$e(Q'_1, V(G) \setminus P) \leq n-|P|$.
\end{claim*}
\begin{proof}
Suppose otherwise. Then there is a vertex $v \in V(G )\setminus P$ with at least $2$ neighbours in $Q'_1$ - denote these by $u,w$. Note that since by construction $u,w$ are at distance at least $2$ and at most $|Q_1| \leq 10n/\delta$ in $P$, this is a contradiction, since it produces the desired $P'$ by substituting the sub-path of $P$ between $u$ and $w$ by the path $uvw$.
\end{proof}
\noindent To give an upper bound on the total number of edges incident to $Q'_1$ which are contained in $V(P)$, we also use the following claim.
\begin{claim*}
For all $i > 1$, we have $e(Q'_1,Q_i) < |Q'_1| + |Q_i|$.
\end{claim*}
\begin{proof}
Suppose otherwise. This implies that there is a cycle in $G[Q'_1,Q_i]$ and hence, there must exist two crossing edges in this bipartite graph, that is, edges $a_1b_1$ and $a_2b_2$, with $a_1 <_P a_2$ and both in $Q_1'$, and $b_1 <_P b_2$ both in $Q_i$. These can clearly be used to shorten $P$ (see Figure \ref{fig:crossingedges}) by at most $|Q_1|+|Q_i| \leq 20n/ \delta$, which is a contradiction as it produces the desired $P'$.
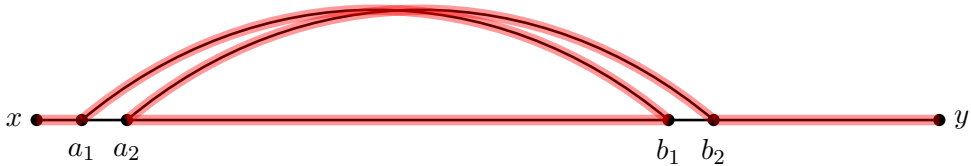
\begin{figure}[b]
\begin{tikzpicture}[scale = 0.6]
\draw[line width=1pt] (-10,0) to (10,0);
\node (y) at (10.5,0) {$y$};
\node[scale = 4] (y) at (10,0) {$.$};
\node[scale = 4] (x) at (-10,0) {$.$};
\node at (-10.5,0) {$x$};
\node[scale = 1] at (-9,-0.7) {$a_1$};
\node at (-8,-0.7) {$a_2$};
\node[scale = 1] at (5,-0.7) {$b_2$};
\node at (4,-0.7) {$b_1$};
\node[scale = 4] (a1) at (-9,0) {$.$};
\node[scale = 4] (a2) at (-8,0) {$.$};
\node[scale = 4] (b2) at (5,0) {$.$};
\node[scale = 4] (b1) at (4,0) {$.$};
\draw[line width=1pt] (-9,0) to [bend left=40](4,0);
\draw[line width=1pt] (-8,0) to [bend left=40](5,0);
\draw[opacity = 0.4, red, line width=4pt] (-9,0) to [bend left=40](4,0);
\draw[opacity = 0.4, red, line width=4pt] (-8,0) to [bend left=40](5,0);
\draw[opacity = 0.4, red, line width=4pt] (-10,0) to (-9,0);
\draw[opacity = 0.4, red, line width=4pt] (-8,0) to (4,0);
\draw[opacity = 0.4, red, line width=4pt] (5,0) to (10,0);
\end{tikzpicture}
    \centering
    \caption{Shortening of the path $P$ using the crossing edges $a_1b_1$ and $a_2b_2$. The resulting path is $P'$ and is drawn in red.}
    \label{fig:crossingedges}
\end{figure}
\end{proof}
The above claim implies that $$\sum_{i > 1} e(Q'_1,Q_i) < \sum_{i > 1} \left(|Q'_1| + |Q_i| \right) \leq (k-1)|Q'_1| + (|P| - |Q_1|) < 2|P| - 2|Q'_1|.$$ To conclude, we now must have the following
$$e(Q'_1,V(G)) = e(Q'_1, Q_1) +  e(Q'_1, V(G) \setminus P) + \sum_{i > 1} e(Q'_1,Q_i) < 2|Q'_1| + (n-|P|) + (2|P| - 2|Q'_1|) < 2n .$$
which contradicts the previous observation that $e(Q'_1,V(G)) \geq 3n$.
\end{proof}
Conversely, the following lemma gives a way to shorten a path using only its independence number. It was proven in \cite{ourpaper} and was used to solve an old conjecture of Erd\H{o}s \cite{erdos1972some} - see Proposition 2.9 in \cite{ourpaper} and let $U=\emptyset$ and $c=\frac{\ceil{20\alpha^2/{|P|}}+3}{4}$.
\begin{lem}\label{lem:jumpwithzigzag}
Let $G$ be an $n$-vertex graph with independence number $\alpha$, let $P$ be a path in $G$ with endpoints $x,y$ such that $|P| > 4 \alpha$. Then there is an 
$xy$-path $P'$ such that $|P| - \lceil 20\alpha^2/|P| \rceil \leq |P'| < |P|$.
\end{lem}
\noindent We finish this section with a lemma which contrarily to the previous lemmas, will allow us to slightly augment a path between two vertices. Further, it will use both the connectivity and the independence number of the graph, and it will be used when the size of the path $P$ we are considering is not suitable to apply the first two lemmas of this subsection.
\begin{lem}\label{lem:augmentingpath}
Let $G$ be an $n$-vertex graph with connectivity $\kappa$ and independence number $\alpha$, and let 
$r\in \mathbb N$. Let $P$ be a path in $G$ with endpoints $x,y$ and with $|P| < n$. Then, there is an $xy$-path $P'$ such that $|P| < |P'| \leq |P| + r$ provided that $|P| > \frac{80 \alpha}{r}$, and $\alpha> r > \frac{80 \alpha}{r} \cdot\max\left(1, \frac{|P|}{\kappa - \alpha } \right)$, while $\kappa>\alpha+2r$.
\end{lem}
\begin{proof}
Consider a vertex $u$ not contained in $P$ and denote $P$ as $v_1v_2 \ldots v_\ell$ with $x=v_1,y=v_\ell$. By Menger's theorem, there exist $\min(\kappa,|P|)$ paths going from $u$ to $V(P)$ which are vertex-disjoint apart from the vertex $u$. 
Let $S \subseteq V(P)$ be the endpoints of these paths, and for each $v_i\in S$ let $P_i$ denote the corresponding path from $u$ to $v_i$.

We first consider the case when $S = V(P)$. Note that if for all $i$, since $v_i,v_{i+1}$ are consecutive in $P$, we can substitute the edge $v_iv_{i+1}$ by the paths $P_i,P_{i+1}$ to form an $xy$-path of length $|P| + |P_i|+|P_{i+1}|$. Hence, if $|P_i|+|P_{i+1}| < r$ for some $i$, then we have constructed the desired $P'$. Otherwise, at least half of the paths $P_i$ with $i\leq \frac{20\alpha}{r}$ have $|P_i| \geq r/2$. Moreover, we can assume that the $P_i$ are induced paths since if not, their length can be shortened. Let $S'$ be the set of vertices $v_i$ which are the endpoints of those paths, and note that $|S'|\geq \frac{10\alpha}{r}$. For each such $P_i$, let $Q_i$ denote the subpath of $P_i$ formed by its $r/4$ vertices in positions $r/4+1,\ldots, r/2$, viewed in the direction $v_i \rightarrow u$. Since $Q_i$ is an induced path, it contains an independent set $I_i$ of size $|Q_i|/2 \geq r/8$. 
Then we have 

$$\left| \bigcup_{v_i \in S'} I_{i} \right| \geq |S'|\frac{r}{8}> \alpha,$$
hence there is an edge $(u_a,u_b)$ between $I_a$ and $I_b$ for some $v_a,v_b\in S'$. This now completes the proof, as we can replace the part of the path in $P$ between $v_a$ and $v_b$ by the path obtained by concatenating the $v_au_a$-path in $P_a$, the edge $u_au_b$ and the $u_bv_b$-path in $P_b$, thus obtaining a path of length at least $|P|+2\cdot r/4-\frac{20\alpha}{r}> |P|$ and at most $|P|+2\cdot\frac{r}{2}$ which completes this case.


Let us now consider the case when $|S| = \kappa$. First we show the following simple claim.
\begin{claim*}
If at least $\alpha+1$ paths $P_i$ are such that $|P_i| < r/2$, then such a $P'$ exists.
\end{claim*}
\begin{proof}
For each one of the endpoints $v_i\in V(P)-\{y\}$ of the paths $P_i$, let $v_i'$ denote its neighbour on $P$ which is closer to $y$. Let $X$ be the set of those at least $\alpha$ vertices, together with the vertex $u$. Then there is an edge between two vertices in $X$. 
This gives an $xy$-path which is strictly longer than $P$, but by at most $r$ (see Fig. \ref{fig:erdoschvatalswitch} for an illustration of an analogous operation).
\end{proof}
By the above claim, we can assume that at least $\kappa - \alpha$ vertices $v_{i_j} \in S$ are such that $|P_{i_j}| \geq r/2$ - and moreover, we can assume that they are induced paths (since otherwise they can be shortened). Let $S'$ be the set of those vertices in $S$, so that $|S'| \geq \kappa - \alpha$. Now, by letting $t=\frac{20\alpha|P|}{r(\kappa-\alpha)}$ we conclude by averaging that $P$ contains an interval $Q$ of length $t$ with at least $\frac{t}{2|P|}(\kappa-\alpha)=\frac{10\alpha}{r}$ vertices in $S'$. By repeating the argument above -- finding the independent sets $I_i\subset P_i$  for each of the $\frac{10\alpha}{r}$ paths $P_i$ which end in $Q$, and then finding an edge between a pair $I_i$ and $I_j$ -- we get a path $P'$ of length at least $|P|-|Q|+2\cdot \frac{r}{4}\geq |P|-t+\frac{r}{2}>|P|$ by our assumption on $r$, and length at most $|P|+2 \cdot \frac{r}{2}$, which completes the last case of the proof.

\end{proof}


\section{Proof of Theorem \ref{thm:main}}\label{sec:thm}
Let $G$ be a graph on $n$ vertices, let $\alpha$ denote its independence number and $\kappa$ its connectivity number. Let $\eps > 0$ and for convenience we may assume that $\varepsilon$ is sufficiently small so that all our calculations go through. Suppose that $n$ is sufficiently large in terms of $\eps$ and that $\kappa \geq (1+ \eps) \alpha$. This immediately implies that $\alpha$ is also sufficiently large in terms of $\eps$ since otherwise, we would have $n \geq 4 (\alpha+1)^4$ which by an old result of Erd\H{o}s \cite{erdos1972some} would already imply pancyclicity.

\subsection*{Upper range: $\min \left(\frac{10^5n}{\varepsilon^2\kappa},\frac{100\alpha}{\varepsilon} \right)$ to $n$}

\noindent We will first construct cycles of all lengths from $m:=\min \left(\frac{10^5n}{\varepsilon^2\kappa},\frac{100\alpha}{\varepsilon} \right)$ to $n$. First, apply Theorem \ref{ECwithtriangle} to $G$, which implies that it contains a $C^{r_1}_n$ with $r_1 =  \eps \alpha/2$. Note that if $m=\frac{10^5n}{\varepsilon^2\kappa}$, then we also have $r_1 \geq \frac{100n}{\kappa}=:r_2$, since in that case $\frac{10^5n}{\varepsilon^2\kappa}\leq \frac{100\alpha}{\varepsilon} $. Hence, in that case $G$ trivially contains $C^{r_2}_\ell$.

Now, let us denote the Hamilton cycle in $C^{r}_n$ by $v_1v_2 \ldots v_nv_1$, with the edges $v_1v_3,v_3v_5, \ldots, v_{2r-1}v_{2r+1}$ present, where $r=r_1$ if $m=\frac{100\alpha}{\varepsilon}$, and $r=r_2$ if $m=\frac{10^5n}{\varepsilon^2\kappa}$. Let $Q$ denote the path $v_1v_2 \ldots v_{2r+1}$, and let $P$ denote the path $v_{2r+1}v_{2r+2} \ldots v_nv_1$. Note that in the subgraph $G[Q]$, the pair $(v_1,v_{2r+1})$ is $0$-dense in the interval $[r,2r]$. We will now show that the same pair is $r/2$-dense in the interval $[m - 2r,n]$ in the graph $G[P]$. Observation \ref{obs:combining} then implies that $G$ contains cycles of all lengths from $m$ to $n$.

In order to show that $(v_1,v_{2r+1})$ is $r/2$-dense in the interval $[m-2r,n]$ in the graph $G[P]$, a simple application of either Lemma~\ref{lem:jumpwithdegree} or Lemma~\ref{lem:jumpwithzigzag} suffices, depending on the where the minimum $m$ is attained. Indeed, let $G' := G[P]$ and note that it has minimum degree at least $\delta'\geq \delta(G) - (2r-1) \geq \kappa - \eps \alpha > (1-\eps)\kappa$ and $\alpha(G') \leq \alpha$. Assume first that $m=\frac{10^5n}{\varepsilon^2\kappa}\leq\frac{100\alpha}{\varepsilon}$, which implies that $20n/\delta'\leq 20n/(1-\eps)\kappa < r/2$. Therefore, we can apply Lemma \ref{lem:jumpwithdegree} to find a $v_{2r+1}v_1$-path $P'$ in $G'$ such that $|P|-r/2 \leq |P| - 20n/\delta' \leq |P'| < |P|$. Further, we can repeat this on $P'$ and continue applying Lemma \ref{lem:jumpwithdegree} in such a manner, until we are left with a path on at most $\frac{10^5n}{\varepsilon^2\kappa}-2r$ vertices.
Note that we can do this, since for every application of the lemma, we will have that the path will be of size at least $\frac{10^5n}{\varepsilon^2\kappa}-2r\geq \frac{10^5n}{\varepsilon^2\kappa}-\frac{200n}{\kappa} \geq 20 n/\delta'$.
This implies that $(v_1,v_{2r+1})$ is $r/2$-dense in the interval $\left[\frac{10^5n}{\varepsilon^2 \kappa}-2r,n \right]$ as desired. 

Assume now that $\frac{10^5n}{\varepsilon^2\kappa}\geq\frac{100\alpha}{\varepsilon}$. Then, we can apply Lemma \ref{lem:jumpwithzigzag} to find a $v_{2r+1}v_1$-path $P'$ in $G'$ such that $|P|-r/2 \leq |P| - \lceil 20 \alpha^2/|P| \rceil \leq |P'| < |P|$. We can repeat this on $P'$ and iteratively apply the same lemma in such a way, until we are left with a path $P_0$ with at most $\frac{100\alpha}{\varepsilon}-2r= \frac{100\alpha}{\varepsilon}-\varepsilon\alpha>\frac{99\alpha}{\varepsilon} $ vertices, so that for all previous paths $P$ in this iteration we have $\lceil 20\alpha^2/|P| \rceil < r/2$. This shows that $(v_1,v_{2r+1})$ is $r/2$-dense in the interval $\left[\frac{100\alpha}{\varepsilon}-2r,n \right]$ as desired.

\subsection*{Lower range: $3$ to $\max(\eps \alpha/2000,n/\alpha)$}

Now we deal with the lower range. Let us first show that $G$ contains the three smallest cycles.
\begin{claim*}
$G$ contains a $C_3$, a $C_4$ and a $C_5$.
\end{claim*}
\begin{proof}
Note that $G$ contains $C_3$ since $\delta(G)\geq \kappa\geq \alpha+1$, so the neighbourhood of a vertex necessarily spans an edge. Suppose now for sake of contradiction that $G$ does not contain a $C_4$. Then, it must be that for every vertex $v$, the graph induced by its neighbourhood $G[N(v)]$ has maximum degree $1$ - indeed, otherwise it contains a path on three vertices, which together with $v$ forms a $C_4$. Moreover, this implies that $N(v)$ contains an independent set $I_v$ of size at least $|N(v)|/2 \geq \kappa/2 \geq (1+\eps)\alpha/2$. Now, take two adjacent vertices $u,v$ in $G$. Since $G$ contains no $C_4$, it must be that $|I_u \cap I_v| \leq 1$ and thus, $(I_u \Delta I_v) \setminus \{u,v\}$ has at least $(1+\eps)\alpha - 3 > \alpha$ vertices. To finish, note that there can be no edge between $I_u\setminus \{v\}$ and $I_v\setminus\{u\}$ since together with $uv$ it would form a $C_4$. Hence, the set $(I_u \Delta I_v) \setminus \{u,v\}$ is an independent set of size larger than $\alpha$, which contradicts the assumption on $G$.

Finally, suppose for sake of contradiction that $G$ contains no $C_5$. Much like before, note that it must be that for every vertex $v$, $G[N(v)]$ has no path on four vertices since this together with $v$ forms a $C_5$. Therefore, $N(v)$ contains an independent set $I_v$ of size at least $|N(v)|/3 \geq \kappa/3 \geq (1+\eps)\alpha/3$. Now, take a vertex $v$, and let $x_1y_1, x_2y_2, x_3y_3$ be disjoint edges contained in $N(v)$ - note these exist since $|N(v)| \geq \kappa \geq \alpha + 7$. Consider also the neighbourhoods $N(x_1),N(x_2), N(x_3)$ and note that they must be disjoint (except for $v$) -- indeed, if e.g., $z \in N(x_1) \cap N(x_2) $ then $vy_1x_1zx_2v$ is a $C_5$ (see Figure \ref{fig:C5} for an illustration). Note also that there cannot exist an edge $zz'$ with $z \in N(x_i), z' \in N(x_j)$ for some $i \neq j$ - indeed, then $vx_izz'x_jv$ is a $C_5$. Concluding, note that it must be that $I_{x_1} \cup I_{x_2} \cup I_{x_3}$ is an independent set and has size at least $|I_{x_1}|+|I_{x_2}|+|I_{x_3}| > \alpha$, which is a contradiction.
\begin{figure}[h]
\begin{tikzpicture}[scale = 0.6]
\node[scale = 3] at (-2,0) {$.$};
\node[scale = 0.9] at (-2.5,0) {$v$};
\node[scale = 3] at (0,-2.5) {$.$};
\node[scale = 3] at (0,-1.5) {$.$};
\node[scale = 3] at (0,-0.5) {$.$};
\node[scale = 3] at (0,0.5) {$.$};
\node[scale = 3] at (0,1.5) {$.$};
\node[scale = 3] at (0,2.5) {$.$};
\draw[line width=1pt] (-2,0) to (0,-2.5);
\draw[line width=1pt] (-2,0) to (0,-1.5);
\draw[line width=1pt] (-2,0) to (0,-0.5);
\draw[line width=1pt] (-2,0) to (0,0.5);
\draw[line width=1pt] (-2,0) to (0,1.5);
\draw[line width=1pt] (-2,0) to (0,2.5);
\draw[line width=1pt] (0,-2.5) to (0,-1.5);
\draw[line width=1pt] (0,-0.5) to (0,0.5);
\draw[line width=1pt] (0,2.5) to (0,1.5);
\node[scale = 0.9] at (0,2.8) {$x_1$};
\node[scale = 0.9] at (0,0.8) {$x_2$};
\node[scale = 0.9] at (0,-1.1) {$x_3$};
\node[scale = 0.9] at (0.5,1.5) {$y_1$};
\node[scale = 0.9] at (0.5,-0.5) {$y_2$};
\node[scale = 0.9] at (0.5,-2.5) {$y_3$};

\node[scale = 3] at (2,1.5) {$.$};
\node[scale = 0.9] at (2.4,1.5) {$z$};
\draw[line width=1pt] (2,1.5) to (0,2.5);
\draw[line width=1pt] (2,1.5) to (0,0.5);
\draw[opacity = 0.3, red, line width=4pt] (-2,0) to (0,1.5);
\draw[opacity = 0.3, red, line width=4pt] (-2,0) to (0,0.5);
\draw[opacity = 0.3, red, line width=4pt] (0,1.5) to (0,2.5);
\draw[opacity = 0.3, red, line width=4pt] (2,1.5)to (0,2.5);
\draw[opacity = 0.3, red, line width=4pt] (2,1.5) to (0,0.5);
\end{tikzpicture}
    \centering
    \caption{An illustration of the cycle $vy_1x_1zx_2v$.}
    \label{fig:C5}
\end{figure}
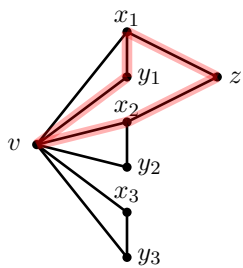
\end{proof}
\noindent For the remaining cycle lengths, let us assume first that $n/\alpha \geq \eps \alpha/2000$, thus implying that $n \geq \eps \alpha^2/2000$. Showing that $G$ contains all cycles of lengths between $6$ and $n/\alpha$ boils down to the study of cycle-complete Ramsey numbers. Namely, the cycle-complete Ramsey number $r(C_\ell,K_s)$ is the smallest number $N$ such that every graph on $N$ vertices either contains a copy of $C_\ell$ or an independent set of size $s$. The following result of Erd\H{o}s, Faudree, Rousseau and Schelp \cite{erdos1978cycle}, along with a more recent result by Keevash, Long and Skokan \cite{keevash2021cycle} cover the mentioned range of cycle lengths we need.

\begin{thm}[\cite{erdos1978cycle}]\label{thm:erdos cycle-complete}
Let $\ell\geq 3$ and $s\geq 2$. Then $r(C_\ell, K_s) \leq\left((\ell-2)(s^{1/x}+2)+1\right)(s-1)$, where $x=\lfloor \frac{\ell-1}{2}\rfloor$.
\end{thm}

\noindent The next result by Keevash, Long and Skokan gives the precise behaviour of cycle-complete Ramsey numbers in a wide range of parameters, and proves a conjecture from \cite{erdos1978cycle}.

\begin{thm}[\cite{keevash2021cycle}]\label{precise cycle-complete}
There exists $C \geq 1$ so that $r(C_\ell , K_s) = (\ell - 1)(s- 1) + 1$ for $s \geq 3$ and $\ell\geq C \frac{\log s}{\log\log s}$.
\end{thm}

\noindent Note that since $G$ contains no independent set of size larger than $\alpha$ and $n \geq \eps \alpha^2/2000$ (and by assumption $\alpha$ is sufficiently large in terms of $\varepsilon$), Theorem~\ref{thm:erdos cycle-complete} implies the existence of a cycle of length $\ell$ for every $\ell\in[6,\log \alpha]$, while Theorem~\ref{precise cycle-complete} covers the range of $[\log \alpha,n/\alpha]$.

\noindent Now assume that $\eps \alpha/2000 > n/\alpha$, implying that $\alpha > 40\sqrt{n/\eps}$. We need to find all cycles from $6$ to $\eps \alpha/2000$. For this, we use the following classic result by Bondy and Simonovits.
\begin{thm}[\cite{BS}]\label{BSthm} Let $G$ be an $n$-vertex graph with $e(G) \geq \max(20ln^{1+1/l},200nl)$. Then, $G$ contains a cycle of length $2l$.
\end{thm}
\noindent We can now utilise this together with Lemma \ref{lem:partition} to get the desired cycle. Indeed, apply this lemma to $G$ to obtain a set $X$ and edge-set $E$ of edges contained in $X$, 
such that $|E|\geq \frac{\kappa-\alpha}{8} \cdot n\geq \varepsilon\alpha n/8$, and for every edge $(x,y)\in E$ there exists $z\in V(G)-X$ such that $x,y$ and $z$ form a triangle. Let $G' := (X,E)$ be the graph consisting of these edges. Observe that it is sufficient for us to show that for all $3 \leq \ell \leq \eps \alpha/2000$, there is a cycle of length $2\ell$ in $G'$ - indeed, such a cycle can then be transformed into a cycle of length $2\ell+1$ in $G$ by substituting an edge $xy$ of the cycle by the path $xzy$ which is guaranteed to exist by Lemma \ref{lem:partition}. Finally, we find these even cycles in $G'$ by applying Theorem \ref{BSthm}, which gives cycles of lengths $2\ell$, for any $\ell$ such that $\max(200nl,20ln^{1+1/l}) \leq \eps \alpha n/8$. Since $\alpha > 40\sqrt{n/\eps}$, this holds for all $\ell \in \left[3,\eps \alpha/2000 \right]$, which completes the proof.

\subsection*{Middle range: $\max \left(\eps \alpha/2000,n/\alpha \right)$ to $\min\left(\frac{10^5n}{\varepsilon^2\kappa},\frac{100\alpha}{\varepsilon}\right)$}
To finish the proof of Theorem \ref{thm:main}, we will now consider the middle range of cycle lengths. First, observe that we may assume that $\max(\eps \alpha/2000,n/\alpha)<\min\left(\frac{10^5n}{\varepsilon^2\kappa},\frac{100\alpha}{\varepsilon}\right)$, as otherwise this range is empty.
Hence we have that $n/\alpha < 100 \alpha/\varepsilon$, which is equivalent to $\alpha > \frac{1}{10}\sqrt{\varepsilon n}$. Further, we have $\eps \alpha/2000 < \frac{10^5n}{\varepsilon^2\kappa}$, and since we have $\kappa>\alpha$, this gives $\alpha < 10^5\sqrt{n/\eps^3}$.  Observe that this implies that $\alpha = \Theta_{\eps}(\sqrt{n})$. 

Now, first observe that by Lemma \ref{lem:cyclewithtriangles}, $G$ contains a $C^{2r}_{\ell}$ with $r = \varepsilon^{10}\alpha = \Theta_{\eps}(\sqrt{n})$ and with $\ell$ such that $$4r+1\leq \ell \leq \frac{n}{\kappa(G)-4r+1}+4r+2\leq\frac{n}{(1+\varepsilon/2)\alpha}+10\varepsilon^{10}\alpha\leq \frac{n}{\alpha},$$
where we used that $10^5\sqrt{n/\eps^3}>\alpha > \frac{1}{10}\sqrt{\varepsilon n}$.

Note that this cycle $C^{2r}_\ell$ can also be viewed as a $C^{r}_\ell$ by omitting some triangles. Let $P$ then be the path consisting of the first $2r+1$ vertices of this $C_\ell^{r}$ (recall that $P$ forms a $P^{r}_{2r}$), and let $P'$ be the other path inside of the cycle with the same endpoints, denoted by $x,y$ - so that $|P'| = l - 2r \geq r$. We will iteratively apply \Cref{lem:augmentingpath} to the path $P'$ inside of the graph $G'=G-(V(P)-\{x,y\})$, with parameter $r$ defined as above, 
and connectivity $\kappa'\geq \kappa-2r$.
Indeed, note that $|P'| \geq r \geq \frac{80 \alpha}{r}$, while $r  \geq \frac{80 \alpha}{r} \cdot \frac{|P'|}{\kappa' - \alpha }$ and $\kappa'>\alpha+2r$ and so, there is an $xy$-path $P''$ in $G'$ with $|P'| < |P''| \leq |P'|+r$. We can continue applying \Cref{lem:augmentingpath} to the newly obtained path inside of the same graph, each time getting a path which is by at most $r$ longer than the previous one. Note that the conditions of the lemma are still satisfied as long as the current path is of length $\frac{100\alpha}{\varepsilon}$. This implies that the pair $xy$ is $r$-dense in $[\ell-2r, 100 \alpha/\eps]$ in the graph $G'$. Now, since $xy$ is also $0$-dense in $[r,2r]$ in $G[P]$, this gives all cycle lengths in $[\ell,100 \alpha/\eps] \supseteq [n/\alpha,100 \alpha/\eps]$ by \Cref{obs:combining}, as desired. \hfill \qedsymbol

\section{Concluding remarks}\label{sec:concludingrem}
In this paper we showed that if a graph $G$ satisfies $\kappa(G)\geq (1+o(1))\alpha(G)$ then $G$ is pancyclic. Moreover, the $o(1)$ error term can be made to be $\alpha(G)^{-c}$ for some small constant $c > 0$. This extends the classic theorem of Chv\'atal and Erd\H{o}s, which states that $\kappa (G)\geq \alpha(G) $ implies that $G$ is Hamiltonian, confirming asymptotically Bondy's meta-conjecture for this celebrated result. Nevertheless, it would be very interesting to prove the Jackson-Ordaz conjecture in full generality, or at least to show that it  holds when $\kappa (G)\geq \alpha(G)+C$ for some constant $C>0$.



\end{document}